\documentclass[leqno,12pt]{article}
\usepackage{amsmath}
\usepackage{amssymb,amscd}
\usepackage{amsthm}
\usepackage{latexsym}

\usepackage{color}  

\newtheorem{thm}{Theorem}[section]
\newtheorem{lem}[thm]{Lemma}
\newtheorem{prop}[thm]{Proposition}
\newtheorem{cor}[thm]{Corollary}
\newtheorem{rem}[thm]{Remark}

\newtheorem{ex}[thm]{Example}

\newcommand{\RI}{R \!\Join\! I}

\newcommand{\R}{\mathcal R_+}
\newcommand{\m}{\mathfrak m}

\newcommand{\Z}{{\mathbb Z}}
\newcommand{\N}{{\mathbb N}}

\newcommand{\Ann}{{\rm Ann}}

\title{A family of quotients of the Rees algebra}
\author{V. Barucci \thanks{{\em email}: barucci@mat.uniroma1.it},
M. D'Anna \thanks{{\em email}: mdanna@dmi.unict.it}, F. Strazzanti
\thanks{{\em email}: strazzanti@mail.dm.unipi.it} }

\date{}

\begin{document}
\maketitle

\begin{center}
{\it Dedicated to Marco Fontana on occasion of his 65-th birthday}
\end{center}

\bigskip

\begin{abstract}
\noindent A family of quotient rings of the Rees algebra
associated to a commutative ring is studied.  This family
generalizes both the classical concept of idealization by Nagata
and a more recent concept, the amalgamated duplication of a ring.
It is shown that several properties of the rings of this family do
not depend on the particular member.
\medskip

\end{abstract}

\section*{Introduction}
Let $R$ be a commutative ring and let $M$ be an $R$-module; the
idealization, also called trivial extension, is a classical
construction introduced by Nagata (see \cite[page 2]{n},
\cite[Chapter VI, Section 25]{H} and \cite{fs}) that produces a
new ring containing an ideal isomorphic to $M$. Recently, D'Anna
and Fontana introduced the so-called amalgamated duplication (see
\cite{DF}, \cite{D'A}, studied also in, e.g., \cite{DF2},
\cite{m-y} and \cite{bsty}), that, starting with a ring $R$ and an
ideal $I$, produces a new ring that, if $M=I$, has many properties
coinciding with the idealization (e.g., they have the same Krull
dimension and if $I$ is a canonical ideal of a local
Cohen-Macaulay ring $R$, both of them give a Gorenstein ring). On
the other hand, while the idealization is never reduced, the
duplication can be reduced, but is never an integral domain.
Looking for a unified approach to these two constructions, D'Anna
and Re in \cite{DR} observed that it is possible to present both
of them as quotients of the Rees algebra modulo particular ideals.
This observation leaded to the subject of this paper, where we
study a more general construction, that produces a ring which, in
some cases, is an integral domain.

More precisely, given a monic polynomial $t^2+at+b \in R[t]$ and
denoting with $\R$ the  Rees algebra associated to the ring $R$ with
respect to the ideal $I$, i.e. $\R =\bigoplus_{n \geq 0}I^nt^n$,
we study the quotient ring $\R/(I^2(t^2+at+b))$, where
$(I^2(t^2+at+b))$ is the contraction to $\R$ of the ideal generated
by $t^2+at+b$ in $R[t]$. We denote such ring by $R(I)_{a,b}$.

%
%

In the first section we introduce the family of rings
$R(I)_{a,b}$, show that idealization and duplication are
particular cases of them (cf. Proposition \ref{isomorphisms}) and
study several general properties such as Krull dimension, total
ring of fractions, integral closure, Noetherianity and  spectrum.
In Section $2$ we assume that $R$ is local; in this case we prove
that the rings $R(I)_{a,b}$ have the same Hilbert function and
that they are Cohen-Macaulay if and only if $I$ is a maximal
Cohen-Macaulay $R$-module. We conclude this section proving that,
if $R$ is a Noetherian integral domain of positive dimension,
there exist infinitely many choices of $b$ such that the ring
$R(I)_{0,-b}$ is an integral domain. Finally in the last section
we study the one-dimensional case. If $R$ is local, Noetherian and
$I$ a regular ideal we find a formula for the CM type of $R(I)_{a,b}$
(cf. Theorem \ref{type}) and prove that it is Gorenstein if and
only if $I$ is a canonical ideal of $R$. Moreover, we show the
connection of the numerical duplication of a numerical semigroup
(see \cite{DS}) with $R(I)_{0,-b}$, where $R$ is a numerical
semigroup ring or an algebroid branch and $b$ has odd valuation
(see Theorems \ref{Semigroup ring} and \ref{Algebroid branch}).

\section{Basic properties}

Let $R$ be a commutative ring with unity and $I$ a proper ideal of $R$;
let $t$ be an indeterminate. The Rees algebra (also called Blow-up
algebra) associated to $R$ and $I$ is defined as the following
graded subring of $R[t]$:
$$\R =\bigoplus_{n \geq 0}I^nt^n \subseteq R[t].$$

\begin{lem}\label{intersection} Let $f(t) \in R[t] $ be a monic polynomial of degree $k$. Then
$$f(t)R[t] \cap \R=\{f(t)g(t); g(t) \in I^k \R\}$$
\end{lem}

\begin{proof} Observe first that $I^k \R=\{\sum_{i=0}^nb_it^i; b_i \in I^{k+i}\}$.
It is trivial that each element of the form $f(t)g(t)$, with $
g(t) \in I^k \R$, is in $f(t)R[t] \cap \R$. Conversely, if $g(t)
\in R[t]$ and if $f(t)g(t) \in \R$, we prove by induction on the
degree of $g(t)$, that $g(t) \in I^k \R$. If the degree of $g(t)$
is zero, i.e. $g(t)=r \in R$, and if $f(t)r \in \R$, then the
leading term of $f(t)r$ is $rt^k$ and $r \in I^k \subset I^k\R$.
The inductive step: suppose that the leading term of $g(t)$ is
$h_nt^n$; thus the leading term of $f(t)g(t)$ is $h_nt^{k+n}$.  If
$f(t)g(t) \in \R$, then $h_n \in I^{k+n}$ and so $f(t)h_nt^n \in
\R$. It follows that, if $f(t)g(t) \in \R$, then $f(t)g(t)
-f(t)h_nt^n=f(t) \bar g(t) \in \R$, where deg$( \bar g(t))<n=\deg
(g(t))$. By inductive hypothesis $\bar g(t) \in I^k\R$, hence
$g(t)=\bar g(t)+h_nt^n \in I^k\R$.
\end{proof}

We denote the ideal of the previous lemma by $(I^kf(t))$.

\begin{lem}\label{banale} Let $f(t) \in R[t] $ be a monic polynomial of degree $k>0$.
Then each element of the factor ring $\R/(I^kf(t))$ is represented
by a unique polynomial of $\R$ of degree $<k$.
\end{lem}

\begin{proof} The euclidean division of an element $g(t)$ of $\R$ by the
monic polynomial $f(t)$ is always possible and gives
$g(t)=f(t)q(t)+r(t)$, with $\deg(r(t))<k$. Moreover, an easy
calculation shows that $q(t) \in I^k\R$ and $r(t) \in \R$. Thus
$g(t) \equiv r(t) \pmod {(I^kf(t))}$. Finally, if $r_1(t)$ and
$r_2(t)$ are distinct polynomials of $\R$ of degree $<k$, also deg
$(r_1(t)-r_2(t))<k$ and they represent different classes.
\end{proof}

It follows from Lemma \ref{banale} that the ring $R$ is a subring
of $\R/(I^kf(t))$.
\begin{prop} \label{dim}
The ring extensions $R \subseteq \R/(I^kf(t)) \subseteq
R[t]/(f(t))$ are both integral and the three rings have the same
Krull dimension.
\end{prop}

\begin{proof}
By the two lemmas above we have  the two inclusions. Moreover, the
class of $t$ in $R[t]/(f(t))$ is  integral over $R$ and over
$\R/(I^kf(t))$ as well. It follows that all the extensions are
integral. By a well known theorem on integral extensions, we get
that the three rings have the same dimension.
 \end{proof}

We observe now that, for particular choices of the polynomial
$f(t)$ above, we get known concepts.

 Recall that
the Nagata's idealization, or simply idealization,
of $R$ with respect to an ideal $I$ of
$R$ (that could be defined for any $R$-module $M$) is defined as
the $R$-module $R\oplus I$ endowed with the multiplication
$(r,i)(s,j)=(rs,rj+si)$ and it is denoted by $R \ltimes I$.

The duplication of $R$ with respect to $I$ is defined as follows:
$$
\RI = \{(r,r+i) \ | \ r \in R, \ i \in I \} \subset R\times R;
$$
note that $\RI \cong R \oplus I$ endowed with the multiplication
$(r,i)(s,j)=(rs,rj+si+ij)$.

\begin{prop} \label{isomorphisms}
We have the following isomorphisms of rings: \\{\rm1)}
$\R/(I^2t^2)\cong R\ltimes I$\\
{\rm 2)} $\R/(I^2(t^2-t))\cong  \RI$
 \end{prop}

\begin{proof} 1) For each residue class modulo $(I^2t^2)$, let $r+it \in \R$,
with $r \in R$ and $i \in I$, be its unique representative; the
map
$$\alpha: \R/(I^2t^2)\rightarrow R\ltimes I$$
defined setting $\alpha(r+it+(I^2t^2))=(r,i)$ is an isomorphism of
rings: as a matter of fact, $\alpha$ preserves sums and, if $r,s
\in R$, $i,j \in I$, we have
$\alpha((r+it+(I^2t^2))(s+jt+(I^2t^2)))=\alpha(rs+(rj+si)t+ijt^2+(I^2t^2))=\alpha(rs+(rj+si)t
+(I^2t^2))=(rs,rj+si)=(r,i)(s,j)$.\\
2) Similarly to 1), the map
$$ \beta:\R/(I^2(t^2-t))\rightarrow \RI$$
defined setting $\beta(r+it+(I^2(t^2-t)))=(r,r+i)$ is an
isomorphism of rings. As for the product, we have $\beta
((r+it+(I^2(t^2-t)))(s+jt+(I^2(t^2-t))))=\beta(rs+(rj+si)t+ijt^2+(I^2(t^2-t)))
= \beta(rs+(rj+si+ij)t
+ij(t^2-t)+(I^2(t^2-t)))=\beta(rs+(rj+si+ij)t  +(I^2(t^2-t))) =
(rs,rs+rj+si+ij)=(r,r+i)(s,s+j)$.
 \end{proof}

 \medskip
The previous proposition makes natural to consider the family
$R(I)_{a,b}=\R/(I^2(t^2+at+b))$, where $a, b \in R$. As $R$-module
$R(I)_{a,b} \cong R\oplus I$ and the natural injection $R
\hookrightarrow R(I)_{a,b}$ is a ring homomorphism; however
$0\oplus I$ in general (if $b\neq 0$) is not an ideal of
$R(I)_{a,b}$, although this happens for idealization and
duplication.

Both idealization and duplication can be realized in other cases.
\begin{prop}\label{othercases}{\rm 1)} If $t^2+at+b=(t-\alpha)^2$,
with $\alpha \in R$, then $R(I)_{a,b}\cong R \ltimes I$.\\ {\rm
2)} If $t^2+at+b=(t-\alpha)(t-\beta)$, with $(t- \alpha)$ and
$(t-\beta)$ comaximal ideals of $R[t]$, then $R(I)_{a,b} \cong R
\!\Join\! I$.
\end{prop}
\begin{proof} 1) It is enough to consider the automorphism of $R[t]$,
induced by $t \mapsto t- \alpha$.\\
2) By Chinese Remainder Theorem, the map $\Phi: R[t]/(t^2+at+b)
\rightarrow R \times R$, defined by $\Phi( \overline {r+st})=(r+
\alpha s,r+ \beta s)$ is an isomorphism of rings, as well as the
map $\Psi: R[t]/(t^2-t) \rightarrow R \times R$, defined by $\Psi(
\overline {r+st})=(r ,r+  s)$, so that $$ \Psi^{-1} \circ
\Phi:R[t]/(t^2+at+b) \rightarrow R[t]/(t^2-t),$$  where
$(\Psi^{-1}\circ \Phi )( \overline {r+st})=\overline{(r+ \alpha s)
+ (\beta- \alpha)st}$ is also an isomorphism.  If we fix an ideal
$I$ of $R$ and if we restrict $ \Psi^{-1} \circ \Phi $ to the
subring $R(I)_{a,b}$, i.e. to the elements $\overline{r+it}$, with
$r \in R$ and $i \in I$, we get
\begin{align*}
R(I)_{a,b} \cong &
\{ \overline{(r+ \alpha i)+(\beta - \alpha)it}; \ r \in R, i \in I\} \\
= & \{\overline{r'+(\beta - \alpha)it}; \ r' \in R, i \in I\};
\end{align*}
the last ring is $R \!\Join\! J$, where $J= (\beta - \alpha)I$.
To finish the proof  we show that  $\beta - \alpha $ is invertible. 
In the authomorphism of $R[t]$ induced by $t \rightarrow t+ \beta$, 
the ideal $(t- \alpha, t- \beta)$ corresponds to 
$(t- \alpha+ \beta,t)=( \beta - \alpha,t)$ and this last ideal 
is $R[t]$ if and only if $\beta - \alpha$ is invertible.
\end{proof}

\begin{ex} \rm
Let $R= \mathbb Z$ and $t^2+at+b=t^2-5t+6=(t-2)(t-3)$. Then 
for each ideal $I$ of $\Z$, $\Z(I)_{-5,6} \cong \Z \!\Join\! I$.
\end{ex}

In this paper we study the family of rings of the form
$R(I)_{a,b}$, showing that many relevant properties are
independent by the member of the family. From now on, we denote
each element of $R(I)_{a,b}$ simply by $r+it$ ($r \in R$, $i\in
I$).

\begin{prop}\label{ringoffractions} Let
$Q$ be the total ring of  fractions of $R(I)_{a,b}$. Then each
element of $Q$ is of the form $r+it \over u$, where $u$ is a
regular element of $R$.
\end{prop}
\begin{proof}
Assume that $(s+jt)$ is a regular element of $R(I)_{a,b}$ and that \\
$(r+it)/(s+jt) \in Q$. Since $(s+jt)$ is regular, then  $x(s+jt)\neq 0$, for every $x \in R\setminus \{0\}$.
 Hence, $xj=0$ implies $xs \neq 0$.

Consider, now, the element $(ja-s+jt)$. To prove the Proposition,
it is enough to show that:

i)  the product $u=(s+jt)(ja-s+jt)$ is a regular element of $R$,

ii) $(ja-s+jt)$ is a regular element of $R(I)_{a,b}$.

In fact in this case we can write $(r+it)/(s+jt)=(r+it)(ja-s-jt)/u$.

Observing that $-at-t^2=b \in R$, we have $u=s(ja-s)-j^2b \in R$.
If $x(ja-s+jt)=0$ (for some $x \in R\setminus \{0\}$), then
$xj=0$, that implies $x(ja-s+jt)=xs=0$, a contradiction. Hence
$(ja-s+jt)$ is not killed by any non zero element of $R$; it
follows that $u$ is regular in $R$, otherwise there would exist
$x\in R\setminus \{0\}$ such that $ux=0$ that implies $(s+jt)$ not
regular in $R(I)_{a,b}$, since it is killed by $(ja-s+jt)x\neq 0$.
Thus i) is proved.

ii): if $(ja-s+jt)$ is not regular in $R(I)_{a,b}$, there exists
$(h+kt)\neq 0$ such that $(ja-s+jt)(h+kt)=0$. Hence $u(h+kt)
=(s+jt)(ja-s+jt)(h+kt)=0$, so $u$ is not regular in $R(I)_{a,b}$.
But then it is not regular in $R$; contradiction.
\end{proof}

\begin{cor} \label{same ring of fractions}
Assume that $I$ is a regular ideal; then the rings $R(I)_{a,b}$ and $R[t]/(t^2+at+b)$
have the same total ring of fractions and the same integral
closure.
\end{cor}
\begin{proof} Each element of $R[t]/(t^2+at+b)$,
let's say $r+r_1t$ with $r,r_1 \in R$, is in $Q$, in fact if $i$
is an element of $I$ regular in $R$, $i$ is also regular in
$R[t]/(t^2+at+b)$ and $r+r_1t=(ir+ir_1t)/ i \in Q$. Moreover, if
$r+r_1t$ is regular in $R[t]/(t^2+at+b)$, it is also regular in
$Q$. In fact, according to Proposition \ref{ringoffractions}, an
element of $Q$ is of the form $(s+jt)/ u$ ($s \in R$, $j \in I$,
$u \in R$ and regular); if $(r+r_1t)(s+jt )/u=0$, then
$(s+jt)/u=0$. It follows that, if $(r+r_1t)/(s+s_1t)$ is an
element of $Q'$, the total ring of fractions of $R[t]/(t^2+at+b)$,
then $r+r_1t$ and $s+s_1t $ belong to $Q$ and  $s+s_1t $ is regular in $Q$,
so $(r+r_1t)/(s+s_1t) \in Q$. On the other hand, if $(r+it)/u\in
Q$, with $u \in R$ and regular in $R$, then $u$ is also regular in
$R[t]/(t^2+at+b)$ and $(r+it)/u\in Q'$.
\end{proof}

By Corollary \ref {same ring of fractions}, it follows that the
integral closure of $R(I)_{a,b}$ contains ${\overline
R}[t]/(t^2+at+b)$, where $\overline R$ is the integral closure of
$R$, but it may be strictly larger. For example, for $R=\Z$ and
$t^2+at+b=t^2+4$, we have that $\Z[t]/(t^2+4)$ is not integrally
closed, in fact $(t/2)^2+1=0$.

\medskip

Using the chain of inclusions $R \subseteq R(I)_{a,b} \subseteq
R[t]/(t^2+at+b)$ and the fact that these extensions are integral,
we can get information on Spec$(R(I)_{a,b})$ with respect to
Spec$(R)$.
\begin{prop} \label{Spec}
For each prime ideal $P$ of $R$, there are at most two prime
ideals of $R(I)_{a,b}$ lying over $P$. Moreover if $t^2+at+b$ is
irreducible on $R/\m$ for any maximal ideal $\m$ of $R$, then
there is exactly one prime ideal of $R(I)_{a,b}$ lying over $P$.
\end{prop}
\begin{proof}
Every prime ideal of $R(I)_{a,b}$, lying over $P$ has to be the
contraction of a prime ideal of $R[t]/(t^2+at+b)$. It is well
known (see e.g. \cite[Chapter 6]{g-s}) that for every prime ideal
$P$ of $R$, $P[t]$ is a prime of $R[t]$ lying over $P$ and there
exist infinitely many other primes in $R[t]$, lying over $P$, all
of them containing $P[t]$ and with no inclusions among them. In
particular, there is a bijection between these ideals and the
nonzero prime ideals of $(Q(R/P))[t]$ (here $Q(R/P)$ denotes the
field of fractions of $R/P$); therefore the image of all these
prime ideals $J$ in $(Q(R/P))[t]$ is of the form $(f(t))$, for
some irreducible polynomial $f(t)$; hence
$J=\varphi_P^{-1}((f(t)))$, where $\varphi_P$ is the composition
of the canonical homomorphisms $R[t] \rightarrow (R/P)[t]
\hookrightarrow (Q(R/P))[t]$. Thus the prime ideals of
$R[t]/(t^2+at+b)$ lying over $P$ are of the form $J/(t^2+at+b)$,
with $J\supseteq (t^2+at+b)$. This means that the polynomial
$f(t)$, corresponding to $J$, divides the image of $t^2+at+b$ in
$Q(R/P)[t]$. Hence, if $t^2+at+b$ is irreducible in $Q(R/P)[t]$,
there is only one prime of $R[t]/(t^2+at+b)$ lying over $P$; on
the other hand, if $t^2+at+b$ has two distinct irreducible factors
in $(Q(R/P))[t]$, there exist exactly two prime ideals in
$R[t]/(t^2+at+b)$ lying over $P$. Hence there are at most two
primes in $R(I)_{a,b}$ lying over $P$ and the first part of the
proposition is proved.

Suppose that $J/(t^2+at+b)) \in$ Spec$(R[t]/(t^2+at+b))$ and
$J/(t^2+at+b)) \cap R=P$. We know that $J=\varphi_P^{-1}((f(t)))$,
where $f(t)$ is an irreducible factor of $t^2+at+b$ in
$Q(R/P)[t]$. If $P' \in$ Spec$(R)$, $P' \subset P$, then the prime
ideals of $R[t]/(t^2+at+b)$ lying over $P'$ correspond to the
irreducible factors of $t^2+at+b$ in $Q(R/P')[t]$; since the
factorization of $t^2+at+b$ in $Q(R/P')[t]$ induces a
factorization in $Q(R/P)[t]$, $f(t)$ is irreducible also in
$Q(R/P')[t]$ and we have a prime ideal of $R[t]/(t^2+at+b)$ lying
over $P'$ of the form $J'/(t^2+at+b)$, with
$J'=\varphi_{P'}^{-1}((f(t))) \subset J$. In particular, if $\m$
is a maximal ideal of $R$ containing $P$ and $t^2+at+b$ is
irreducible on $R/\m$, then there is one and only one prime ideal
of $R[t]/(t^2+at+b)$ lying over $P$ and the same happens for
$R(I)_{a,b}$  because the extension $R(I)_{a,b} \subseteq
R[t]/(t^2+at+b)$ is integral.
\end{proof}

\begin{rem}\label{Spec2}\rm
1) Notice that, for particular $a$ and $b$, the factorization of
$t^2+at+b$ in $Q(R/P)[t]$ may not depend on $P$. For example, in
the case of the idealization, the equality $t^2=t\cdot t$, implies
that there is only one prime lying over $P$, both in $R[t]/(t^2)$
and in the idealization. As for the case of the duplication, the
equality $t^2-t=t\cdot (t-1)$, implies that there are two primes
in $R[t]/(t^2-t)$ lying over $P$, namely $(P,t)$ and $(P, t-1)$.
Contracting these primes to the duplication we get the same prime
if and only if $P \supseteq I$ (see, e.g., \cite{DF}).

2) By the proof of Proposition \ref{Spec} we see that the
extension  $R \subseteq R[t]/(t^2+at+b)$ and the extension $R
\subseteq R(I)_{a,b}$ as well fulfill the going down property. In
particular a minimal prime of $R(I)_{a,b}$ lies over a minimal
prime $P$ of $R$.

3) The proof of the previous proposition also implies that a sufficient
condition for $R(I)_{a,b}$ to be an integral domain is that $R$ is
an integral domain and $t^2+at+b$ is irreducible in $Q(R)[t]$. We will see
in the next section that, under particular assumptions on $R$, we
can prove the existence of such polynomials.
%
\end{rem}

We conclude this section characterizing  the rings
$R(I)_{a,b}$ which are Noetherian.

\begin{prop} The following conditions are equivalent:\\
 {\rm (i)} $R$ is a Noetherian ring;\\
 {\rm (ii)} $R(I)_{a,b}$ is a Noetherian ring for all $a,b \in
 R$;\\
 {\rm (ii)} $R(I)_{a,b}$ is a Noetherian ring for some $a,b \in R$.

\end{prop}
\begin{proof}
If $R$ is Noetherian, also the Rees algebra $\R$ is Noetherian;
hence it is straightforward that $R(I)_{a,b}$ is Noetherian for
every $a,b \in R$, being a quotient of a Noetherian ring.

Since the condition (iii) is a particular case of (ii), we need to
prove only that (iii) implies (i). Assume by contradiction that
$R$ is not a Noetherian ring; then there exists an ideal $J=(f_1,
f_2, \dots )$ of $R$ that is not finitely generated and we can
assume that $f_{i+1} \notin (f_1, \dots f_i)$ for any $i$.
Consider the ideal $JR(I)_{a,b}$ of $R(I)_{a,b}$; by hypothesis,
it is finitely generated and its generators can be chosen from
those of $J$ (regarded as elements of $R(I)_{a,b}$). Hence we can
assume that $JR(I)_{a,b}=(f_1, \dots, f_s)$. This implies
$f_{s+1}= \sum_{k=1}^s f_k(r_k + i_k t)$, for some $r_k \in R$ and
$i_k \in I$, and therefore $f_{s+1}=\sum_{k=1}^s f_k r_k$;
contradiction.
\end{proof}

\section{The local case}

Assume that $R$ is local, with maximal ideal $\m$. Then it is
known that both $\RI$ and $R\ltimes I$ are local with maximal
ideals $\m \oplus I$ (in the first case under the isomorphism $\RI
\cong R \oplus I$). More generally:

\begin{prop}
$R$ is local if and only if $R(I)_{a,b}$ is  local. In this case
the maximal ideal of $R(I)_{a,b}$ is $\mathfrak m \oplus I$ (as
$R$-module).
\end{prop}
\begin{proof} Let $R$ be local; we claim that all the elements
$r+it$ with $r \notin \m$ are invertible in $R(I)_{a,b}$. As a
matter of fact, looking for $s+jt$ such that $(r+it)(s+jt)=1$, we
obtain the linear system
$$
\left\{\aligned &rs-ibj=1 \\
&is+(r-ia)j=0
\endaligned
\right.
$$
which has determinant $\delta= r^2-iar+i^2b \in r^2+  \m$. Thus
$\delta$ is invertible in $R$; moreover, it is easy to check that
if $(s,j)$ is the solution  of the system, then $j \in I$; hence
$s+jt \in R(I)_{a,b}$ and it is the inverse of $r+it$.

Conversely, if $R(I)_{a,b}$ is local, $R$ has to be
local, since $R\subseteq R(I)_{a,b}$ is an integral extension (cf. Proposition \ref{dim}).
\end{proof}
It is also clear that, if $(R,\m)$ is local and if we denote by
$M$ the maximal ideal of $R(I)_{a,b}$, then $k=R/\m \cong
R(I)_{a,b}/M$. In the sequel, we will always denote with $k$ the
common residue field of $R$ and $R(I)_{a,b}$.

\begin{rem} \rm \label{length}
Since $R(I)_{a,b}$ is an $R$-algebra, every $R(I)_{a,b}$-module
$N$ is also an $R$-module and then $\lambda_{R(I)_{a,b}}(N) \leq
\lambda_R(N)$ (where $\lambda(\_ \ )$ denote the length of a
module).

If we consider a $R(I)_{a,b}$-module $N$ annihilated by $M$, we
have that, as $R$-module, $N$ is annihilated by $\m$. Hence it is
naturally an $R(I)_{a,b}/M$-vector space and an $R/\m$-vector
space; in particular,
$\lambda_{R(I)_{a,b}}(N)=\text{dim}_k(N)=\lambda_R(N)$ (where
$k=R/\m \cong R(I)_{a,b}/M$).
\end{rem}

For a Noetherian local ring $(R, \m)$, we denote by $\nu(I)$ the
cardinality of a minimal set of generators of the ideal $I$. The
embedding dimension of $R$, $\nu(R)$, is by definition $\nu(\m)$
and the Hilbert function of $R$ is
$H(n)=\lambda_R(\m^n/\m^{n+1})$.

\begin{prop} \label{Hilb} Let $(R,\m)$ be a Noetherian local ring. Then, for every $a,b \in R$
the rings $R(I)_{a,b}$ have the same Hilbert function. In
particular,   they have the same embedding dimension
$\nu(R(I)_{a,b})=\nu(R)+\nu(I)$ and the same multiplicity.
\end{prop}

\begin{proof} First of all, let us consider $M^2$; we have $M^2=\m^2+\m It$ (and hence, as $R$-module,
it is isomorphic to $\m^2 \oplus \m I$): in fact, if $(r+it)$ and
$(s+jt)$ are in $M$, then their product $rs-bij+(rj+si-aij)t \in
\m^2 \oplus \m I$. Conversely, pick an element in $\m^2 \oplus \m
I$ of the form $rs+uit$ (with $r,s,u \in \m$ and $i \in I$); we
have $rs+uit=rs+u(it) \in M^2$; since $\m^2 \oplus \m I$ is
generated by elements of this form we have the equality. Arguing
similarly for any $n \geq 2$, we immediately obtain that
$M^n=\m^n+\m^{n-1}I$ and, as $R$-module, it is isomorphic to
$\m^n\oplus \m^{n-1}I$.

It follows that, as  $R$-modules, $M^n/M^{n+1} \cong \m^n/\m^{n+1}
\oplus I\m^{n-1}/I\m^n$. By the previous remark the length of
$M^n/M^{n+1}$ as $R(I)_{a,b}$-module coincides with its dimension
as $k$-vector space and with its length as $R$-module. The thesis
follows immediately.
\end{proof}

\begin{rem} \rm
Let $R$ be a Noetherian local ring. By Propositions \ref{dim} and
\ref{Hilb} we get
$$\dim R(I)_{a,b} = \dim R \leq \nu(R) \leq \nu(R) + \nu(I) = \nu(R(I)_{a,b}).$$
The first inequality is an equality if and only if $R$ is regular
and the second if and only if $\nu(I)=0$, that is equivalent to
$I=0$, by Nakayama's lemma. This means that $R(I)_{a,b}$ is
regular if and only if $R$ is regular and $I=0$; clearly if $I=0$
one has $R(I)_{a,b}=R$.
\end{rem}

We want to show that, if $R$ is a local Noetherian integral domain, we can
always find integral domains in the family of rings $R(I)_{a,b}$.
The following proposition was proved in \cite{DR} and we publish
it with the permission of the second author.

\begin{prop}\label{prop:x^n-b}
Let $(R,\m)$ be a local Noetherian integral domain with $\dim R\geq 1$ and
$Q(R)$ its field of fractions. Then for any integer $n>1$, not multiple
of $4$, there exist infinitely many elements $b\in R$ such that the
polynomial $t^n-b$ is irreducible over $Q(R)$.
\end{prop}

\begin{proof}
We will use the following well-known criterion of irreducibility:
if $b$ is not a $p$-th power for any prime $p|n$ and
$b\not\in-4Q(R)^4$ if $4|n$, then $t^n-b$ is irreducible (see
\cite[Chapter VI, Theorem 9.1]{lang:algebra}). In particular, if
$4$ does not divide $n$ and $b$ is not a $d$-th power for any
integer $d>1$ such that $d|n$, then $t^n-b$ is irreducible.

Taking a prime ideal $P\subset R$ such that $\mathrm{ht}\, P=1$ we
have $\dim R_P=1$ and, by the Krull-Akizuki Theorem, its integral
closure $\overline{R_P}$ of $R_P$ in $Q(R_P)=Q(R)$ is Noetherian
(see, e.g. \cite[Theorem 4.9.2]{h-s}), hence it is a Dedekind
ring. So there is at least a discrete valuation $v:
Q(R)^*\rightarrow \mathbb{Z}$ with
$v((\overline{R_P})_{M})=\mathbb{N}$ (with $M$ maximal ideal of
$\overline{R_P}$). Since $R \subseteq R_P \subseteq
(\overline{R_P})_M$ have the same field of fractions, it follows
that $v(R)\subseteq \mathbb{N}$ is a semigroup containing two
consecutive integers; so there exists $c>0$ such that any
$x\in\mathbb{N}$, $x\geq c$ belongs to $v(R)$.

In particular, there exist infinitely many elements $b\in R$ such
that $v(b)$ is prime to $n$, so $b$ cannot be a $d$-th power in
$Q(R)$ for any $d>1$ such that $d|n$. Hence we can find infinitely
many $b\in R$ such that $(t^n-b)\subset Q(R)[t]$ is irreducible.
\end{proof}

\begin{cor}
Let $R$ be a local Noetherian integral domain with $\dim R\geq 1$,
let $Q(R)$ be its field of fractions and let $I\subset R$ be an
ideal. Then there exist infinitely many elements $b\in R$ such
that $R(I)_{0,-b}$ is an integral domain.
\end{cor}
\begin{proof}
By Proposition \ref{prop:x^n-b} we can find $b$ such that
$(t^2-b)$ is irreducible in $Q(R)[t]$. The thesis now follows by
point 3) of Remark \ref{Spec2}.
\end{proof}

Now we want to investigate the Cohen-Macaulayness of $R(I)_{a,b}$.

Assume that $R$ is a CM ring; we set $\dim R=\mathrm{depth}\, R=d$;
moreover, $\Ann(R(I)_{a,b})=(0)$, hence the dimension of
$R(I)_{a,b}$ as $R$-module (i.e., since $R(I)_{a,b}$ is a finite
$R$-module, $\dim(R/\Ann(R(I)_{a,b}))$) equals the Krull dimension of
$R(I)_{a,b}$. We can assume that $d \geq 1$, otherwise both $R$
and $R(I)_{a,b}$ are trivially CM.

Given a regular sequence ${\bf x}=x_1, x_2, \dots , x_d$ of the
ring $R$, it is not difficult to
check that it is an $R(I)_{a,b}$-regular sequence if and only if
its image in $R(I)_{a,b}$ is a regular sequence of $R(I)_{a,b}$ as
a ring. Moreover, since ${\bf x}$ is a system of parameters of
$R$, then it is a system of parameters of $R(I)_{a,b}$ (since
$R\subseteq R(I)_{a,b}$ is an integral extensions) and ${\bf x}$
is a system of parameters for the $R$-module $R(I)_{a,b}$. Hence,
arguing exactly as in \cite{D'A} we have that $R(I)_{a,b}$ is a CM
ring if and only if it is a CM $R$-module.

Since $R(I)_{a,b} \cong R \oplus I$ as $R$-module, it follows that
$\mathrm{depth}\,(R\oplus I)=$ min$\{\mathrm{depth}\,I, \mathrm{depth}\, R\}= \mathrm{depth}\, I$
and therefore $R(I)_{a,b}$ is a CM $R$-module if and only if $I$
is a CM $R$-module of dimension $d$ (that is if and only if $I$ is
a maximal CM $R$-module).

Hence we can state the following:

\begin{prop}\label{CM}
Assume that $R$ is a local CM ring of dimension $d$. Then
$R(I)_{a,b}$ is CM if and only if $I$ is a CM $R$-module of
dimension $d$. In particular, the Cohen-Macaulayness of
$R(I)_{a,b}$ depends only on the ideal $I$.
\end{prop}

\begin{rem}\rm
We notice that if $I$ is a canonical ideal of $R$, since
$R(I)_{a,b}\cong R\oplus I$, we can apply a result of Eisenbud
(stated and proved in \cite{D'A}) to get that $R(I)_{a,b}$ is
Gorenstein for every $a, b \in R$.
We will see that in the one-dimensional case we can determine the
CM type of $R(I)_{a,b}$ and deduce that it is a Gorenstein ring if
and only if $I$ is a canonical ideal.
\end{rem}

\begin{rem} \rm
In \cite[Corollary 5.8]{DFF}, under the assumption that the ring
$(R,\mathfrak m)$ is a local CM ring with infinite residue field,
it has been proved the following formula about the multiplicity of
the duplication: $e(R\!\Join \!\! I)=e(R)+\lambda_R(I/IJ)$ (where
$J$ is any minimal reduction of $\mathfrak m$); in particular, if
$\dim R=1$, then $e(R\!\Join \!\! I)$ $=2e(R)$.

By Proposition \ref{Hilb} we can state that, under the same
assumptions, the same formulas hold for the multiplicity of
$R(I)_{a,b}$, for every $a,b \in R$.
\end{rem}

\section{One-dimensional case}
Assume for all this section that $(R,\mathfrak{m})$ is a one-dimensional,
Noetherian, and local ring and $I$ a regular ideal; in this
section we determine the CM type of $R(I)_{a,b}$.



Since $I$ is regular, it is a maximal Cohen-Macaulay $R$-module
and $R$ is a CM ring; therefore $R(I)_{a,b}$ is also CM by Proposition
\ref{CM}. In this case the type of $R(I)_{a,b}$ equals the length
of $(R(I)_{a,b}:M)/R(I)_{a,b}$ as $R(I)_{a,b}$-module, where $M$
is the maximal ideal of $R(I)_{a,b}$; so we start studying
$R(I)_{a,b}:M$.

\begin{lem}
For any $a,b \in R$, the $R(I)_{a,b}$-module $R(I)_{a,b}:M$ is equal to
 $$\left\{\frac{r}{s}+\frac{i}{s} \ t; \ \frac{i}{s} \in I:\m, \frac{r}{s} \in (I:I) \cap (R:\m)\right\}$$
\end{lem}

\begin{proof}
Consider a generic element $r/s+(i/s)t$ of $Q(R(I)_{a,b})$,
where $r,s \in R, i \in I$ and $s$ is regular (cf. Proposition \ref{ringoffractions}). It is an element
of $R(I)_{a,b}:M$ if and only if
$$
\begin{array}{ll}
(r/s+(i/s)t)(m+jt)&=rm/s+(im/s)t+(rj/s)t+(ij/s)t^2\\
&=rm/s-ijb/s+(im/s+rj/s-ija/s)t
\end{array}
$$
is an element of $R(I)_{a,b}$, for any $m \in \m$ and for any $j
\in I$, that is $(rm/s-ijb/s) \in R$ and $(im/s + rj/s -ija/s) \in
I$.

Suppose that $r/s+(i/s)t \in R(I)_{a,b}:M$; in particular, if
$j=0$ we have $rm/s \in R$ and $im/s \in I$, that is $r/s \in
R:\m$ and $i/s \in I:\m$. Moreover since $ja \in I \subseteq \m$
and $i/s \in I:\m$, we have $im/s,ija/s \in I$, hence $rj/s \in I$
for any $j \in I$ and then $r/s \in I:I$.

Conversely, suppose that $i/s \in I: \m$ and $r/s \in (I:I) \cap
(R:\m)$. Then $rm/s-ijb/s \in R+I=R$ and $im/s +rj/s - ija/s \in
I+I+I=I$, consequently $r/s+(i/s)t \in R(I)_{a,b}:M$.
\end{proof}

\begin{thm} \label{type}
The CM type of $R(I)_{a,b}$ is
$$
t(R(I)_{a,b})=\lambda_R\left(\frac{(I:I)\cap
(R:\m)}{R}\right)+\lambda_R\left(\frac{I:\m}{I}\right);
$$
in particular, it does not depend on $a$ and $b$.
\end{thm}

\begin{proof}
Consider the homomorphism $\varphi$ of $R$-modules
$$
\aligned
R(I)_{a,b}:M &\rightarrow \frac{(I:I) \cap (R:\m)}{R} \times \frac{I:\m}{I} \\
\frac{r}{s}+\frac{i}{s}t &\mapsto \left(\frac{r}{s}+R, \ \frac{i}{s}+I\right).
\endaligned
$$
Thanks to the previous lemma, $\varphi$ is well defined and
surjective; moreover, its kernel is given by the elements
$r/s+(i/s)t$ with $r/s \in R$ and $i/s \in I$, that is $\ker
\varphi=R(I)_{a,b}$; hence
$$
\frac{R(I)_{a,b}:M}{R(I)_{a,b}} \cong \frac{(I:I) \cap (R:\m)}{R} \times \frac{I:\m}{I}.
$$
Consequently, using Remark \ref{length}, we have

\begin{align*}
t(R(I)_{a,b})= &
\lambda_{R(I)_{a,b}}\left(\frac{R(I)_{a,b}:M}{R(I)_{a,b}}\right)=
\lambda_R\left(\frac{R(I)_{a,b}:M}{R(I)_{a,b}}\right)= \\
=& \lambda_R\left(\frac{(I:I) \cap (R:\m)}{R} \times \frac{I:\m}{I}\right)= \\
=&\lambda_R\left(\frac{(I:I) \cap (R:\m)}{R}\right)+\lambda_R\left(\frac{I:\m}{I}\right).
\end{align*}
\end{proof}

\begin{cor} \label{Gor}The ring $R(I)_{a,b}$ is Gorenstein if and only if $I$ is a canonical ideal of $R$.
\end{cor}

\begin{proof} Recall first that a ring is Gorenstein if and only if it has CM type 1.
Recall also that $I$ is a canonical ideal of a one-dimensional CM
local ring $R$, i.e. an ideal $I$ such that $I:(I:J)=J$ for each
regular ideal $J$ of $R$, if and only if $\lambda_R((I:\m)/I)=1$
(cf. \cite {h-k}, Satz 3.3). Notice  that, for any ideal $I$
regular and proper, $\lambda_R((I:\m)/I)\geq 1$.

Thus, by the formula of Theorem \ref{type} we get: $R(I)_{a,b}$ is
Gorenstein if and only if $t(R(I)_{a,b})=1=0+1$; hence,
$\lambda_R((I:\m)/I)=1$, i.e. $I$ is a canonical ideal.
Conversely if $I$ is a canonical ideal, then $I:I=R$ and
$\lambda_R((I:\m)/I)=1$; by the same formula we get
$t(R(I)_{a,b})=0+1=1$, i.e. $R(I)_{a,b}$ is Gorenstein.
\end{proof}

We conclude this section studying two particular cases of one
dimensional rings: numerical semigroup rings and algebroid
branches; in both cases we show the connection with the numerical
duplication of  a numerical semigroup (see \cite{DS}).

Recall that a numerical semigroup is  a submonoid of $\N$ with
finite complement in $\N$ and it can be expressed in terms of its
minimal set of generators, $S= \langle n_1, \dots,n_\nu \rangle$,
with GCD$(n_1,\dots, n_\nu)=1$. A semigroup ideal $E$ is a subset
of $S$ such that $S+E \subseteq E$. We set $2 \cdot E= \{2s| \ s
\in E\}$. According to \cite{DS}, the numerical duplication of $S$
with respect to a semigroup ideal $E$ of $S$ and an odd integer
$m\in S$ is the numerical semigroup
$$S\!\Join^m\!\!E=2\cdot S \cup (2\cdot E+m).$$
A numerical semigroup ring is a ring $R$ of the form
$k[[S]]=k[[X^{n_1}, \dots,X^{n_\nu}]]$, where $k$ is
a field and $X$ an indeterminate. Such a ring is a
one-dimensional, Noetherian, local integral domain; moreover,
it is analytically irreducible (i.e. its integral closure
$\overline R$ is a DVR, which is a finite $R$-module) and
in this case ${\overline R}=k[[X]]$, the ring of formal power
series. The valuation $v$ induced by $k[[X]]$ on $k[[S]]$ is given
by the order of a formal power series and, if $I$ is an ideal of
$k[[S]]$, $v(I)=\{ v(i); \ i \in I, i \neq 0 \}$ is a semigroup
ideal of $S$.

\begin{thm} \label{Semigroup ring}
Let $R=k[[S]]$ be a numerical semigroup ring, let
$b=X^m \in R$, with $m$ odd, and let $I$ be a proper ideal of
$R$. Then $R(I)_{0,-b}$ is isomorphic to the semigroup ring
$k[[T]]$, where $T=S\!\Join^m\!\!v(I)$.
\end{thm}

\begin{proof} If $S= \langle n_1, \dots,n_\nu \rangle$, an element of
$R(I)_{0,b}$ is of the form $$r(X)+i(X)t$$ where $r(X)=r(X^{n_1},
\dots,X^{n_\nu}) \in k[[S]]$ and $i(X)=i(X^{n_1}, \dots,X^{n_\nu})
\in I$. Taking into account that we are factoring out the ideal
$(I^2(t^2-X^m))$, we can easily check that the map $\Phi:
R(I)_{0,-b} \rightarrow k[[T]]$, defined by $$\Phi (r(X)+i(X)t)=
r(X^2)+i(X^2)X^m \ ,$$ is an isomorphism of rings.
\end{proof}

\begin{ex} \rm If $R=k[[X^2,X^3]]$, $b=X^5$ and $I= X^3k[[X^2,X^3]]$,
then $R(I)_{0,b} \cong k[[X^4, X^6,X^{11}]]$. According to
Corollary \ref{Gor}, we get a Gorenstein ring (in fact the
semigroup $\langle 4,6,11\rangle$ is symmetric), because the ideal
$I$ is a canonical ideal of $R$.
\end{ex}

We consider now the case of algebroid branches, i.e. local rings
$(R, \mathfrak m)$ of the form $k[[X_1,\dots X_n]]/P$, where $P$
is a prime ideal of height $n-1$ and $k$ is algebraically closed.
We have that $(R,\mathfrak m)$ is a one-dimensional, Noetherian,
complete, local integral domain; moreover, $R$ is analytically
irreducible with integral closure isomorphic to $k[[X]]$ and
$k\subset R$. If we consider the valuation $v$ induced by $k[[X]]$
on $R$, we get again that $v(R)=\{ v(r); \ r\in R, r \neq 0 \}$ is
a numerical semigroup and that $v(I)=\{ v(i); \ i \in I, i \neq 0
\}$ is a semigroup ideal of $v(R)$.

\begin{thm} \label{Algebroid branch}
Let $R$ be an algebroid branch and let $I$ be a proper ideal of
$R$; let $b \in R$, such that $m=v(b)$ is odd. Then $R(I)_{0,-b}$
an algebroid branch and its value semigroup is
$v(R)\!\Join^m\!\!v(I)$.
\end{thm}

\begin{proof} Since $v(b)$ is odd, by Proposition \ref{prop:x^n-b}, $t^2-b$ is irreducible in $Q(R)[t]$ and
$R(I)_{0,-b}$ is an integral domain. Moreover, applying the
results of the previous sections, we know that $R(I)_{0,-b}$ is
local (we will denote by $M$ its maximal ideal), Noetherian and
one-dimensional. It is not difficult to check that the $\mathfrak
m$-adic topology on the $R$-module $R(I)_{0,-b}$ coincide with the
$M$-adic topology, hence it is complete. Since $R(I)_{0,-b}$
contains its residue field $k$, by Cohen structure theorem, it is
of the form $k[[Y_1,\dots,Y_l]]/Q$, for some prime ideal $Q$ of
height $l-1$; so it is an algebroid branch.

Let $V=k[[Y]]$ be the integral closure of $R(I)_{0,-b}$ in its
quotient field $Q(R(I)_{0,-b})=Q(R)(t)=k((Y))$. We denote by $v'$
the valuation associated to $k[[Y]]$; in particular $v'(Y)=1$.
Since $Q(R)=k((X))$, we have $k((Y))=k((X))(t)$; moreover, $t^2=b$
implies that $2v'(t)=v'(b)=mv'(X)$. In order to obtain $v'(Y)=1$
it is necessary that $v'(t)=m$ and $v'(X)=2$.

Now, it is straightforward that
$v'(R(I)_{0,-b})=v(R)\!\Join^m\!\!v(I)$.
\end{proof}

\end{document}